%% file: paper.tex
\documentclass{amsart}

\usepackage[numbers]{natbib}
\usepackage[foot]{amsaddr}
\usepackage{amsmath,amsthm,amssymb,bbm}
\usepackage{enumitem}
\usepackage[T1]{fontenc}
\usepackage[left=1.5in,right=1.5in,top=1.5in,bottom=1.5in]{geometry}
\usepackage[colorlinks=true,linkcolor=darkgray,citecolor=darkgray,urlcolor=darkgray]{hyperref}
\usepackage[utf8]{inputenc}
\usepackage{url}
\usepackage{xcolor}
\usepackage{array}
\usepackage[belowskip=6pt]{caption}
\usepackage{graphicx}
\usepackage{tikz}
\usepackage{mathtools}

\theoremstyle{plain}
\newtheorem{thm}{Theorem}
\newtheorem{prop}[thm]{Proposition}
\newtheorem{lem}[thm]{Lemma}
\newtheorem{cor}[thm]{Corollary}
\newtheorem{defn}[thm]{Definition}

\newcommand{\N}{\mathbb{N}}
\newcommand{\R}{\mathbb{R}}
\newcommand{\Z}{\mathbb{Z}}
\newcommand{\h}{\mathbb{H}}
\newcommand{\T}{\mathbb{T}}

\newcommand{\conv}{\operatorname{conv}}
\newcommand{\interior}{\operatorname{int}}
\newcommand{\bd}{\operatorname{bd}} 
\newcommand{\vol}{\operatorname{vol}}
\newcommand{\flt}{\operatorname{\phi}}
\newcommand{\orderO}{\mathcal{O}}
\newcommand{\orderOmega}{\Omega}
\newcommand{\orderTheta}{\Theta}
\newcommand{\cP}{\mathcal{P}}
\newcommand{\vertices}{\operatorname{vert}}
\newcommand{\facets}{\operatorname{fct}}
\newcommand{\vx}{\operatorname{vert}}

\newcommand{\floor}[1]{\left\lfloor #1 \right\rfloor}
\newcommand{\ceil}[1]{\left\lceil #1 \right\rceil}
\newcommand{\sprod}[2]{\left< #1 , #2 \right>}

\newcommand{\aff}{\operatorname{aff}}

\newcommand{\setcond}[2]{\left\{ #1 : #2 \right\}}
\newcommand{\nonverts}[2]{#2 \cap #1 \setminus \vertices(#1)}

\newcommand{\ac}{\operatorname{ac}}
\DeclarePairedDelimiter\card{\lvert}{\rvert}
\newcommand{\nverts}[1]{\card{\vertices(#1)}}

\begin{document}

\title{Tight bounds on discrete quantitative Helly numbers}
\date{\today}

\author[Averkov]{Gennadiy Averkov$^1$}
\author[Gonz\'alez Merino]{Bernardo Gonz\'alez Merino$^2$}
\author[Henze]{Matthias Henze$^3$}
\author[Paschke]{Ingo Paschke$^1$}
\author[Weltge]{Stefan Weltge$^1$}

\address{$^1$Otto-von-Guericke-Universität Magdeburg, Germany}
\address{$^2$Technische Universität München, Germany}
\address{$^3$Freie Universität Berlin, Germany}

\email{averkov@ovgu.de}
\email{bg.merino@tum.de}
\email{matthias.henze@fu-berlin.de}
\email{ipaschke@gmx.net}
\email{weltge@ovgu.de}

\thanks{$^2$The second author is partially supported by Consejer\'ia de Industria,
Turismo, Empresa e Innovaci\'on de la CARM through Fundaci\'on S\'eneca, Agencia de Ciencia y Tecnolog\'ia de la Regi\'on de
Murcia, Programa de Formaci\'on Postdoctoral de Personal Investigador and Fundaci\'on S\'eneca project 19901/GERM/15, and MINECO project reference MTM2015-63699-P, Spain.
$^3$The third author was supported by the Freie Universität Berlin within the Excellence Initiative of the
German Research Foundation.}

\begin{abstract}
    Given a subset $ S $ of $ \R^n $, let $ c(S, k) $ be the smallest number $ t $ such that whenever finitely many convex sets have exactly $ k $ common points in $ S $, there exist at most $ t $ of these sets
    that already have exactly $ k $ common points in $ S $.
    For $ S = \Z^n $, this number was introduced by \citeauthor{AlievBDL14}~\citeyearpar{AlievBDL14} who
    gave an explicit bound showing that $ c(\Z^n,k) = \orderO(k) $ holds for every fixed $ n $.
    Recently, \citeauthor{ChestnutHZ15}~\citeyearpar{ChestnutHZ15} improved this to $ c(\Z^n,k) = \orderO(k \cdotp (\log
    \log k) \cdot (\log k)^{-1/3} ) $ and provided the lower bound $ c(\Z^n,k) = \orderOmega(k^{(n-1)/(n+1)}) $.

    We provide a combinatorial description of $c(S,k)$ in terms of polytopes with vertices in $S$ and use it to improve the previously known bounds as follows: We strengthen the bound of \citeauthor{AlievBDL14}~\citeyearpar{AlievBDL14} by a constant factor and extend it to general discrete sets~$S$. 
    We close the gap for~$ \Z^n $ by showing that $c(\Z^n,k) = \orderTheta(k^{(n-1)/(n+1)})$ holds for every fixed $n$. Finally, we determine the exact values of $ c(\Z^n,k) $ for all $k \le 4$.
\end{abstract}

\maketitle

\input{sec-intro}
\input{sec-polytopal}
\input{sec-upper-bound-general}
\input{sec-bounds-integers}
\input{sec-specific-values}

\appendix
\input{sec-maximal-sets}
\input{sec-appendix}
\vspace{1em}
\noindent\textit{\bf Acknowledgements.} We thank Imre B\'{a}r\'{a}ny, Daniel Dadush, Jes\'{u}s A.~De Loera, and the authors of~\cite{ChestnutHZ15} for sharing their expertise. The second author would like to thank Martin Henk
for his invitation to the Technische Universit\"at Berlin in November 2015, where fruitful discussions
on the topic of this paper were carried out together with the third author.

\bibliographystyle{plainnat}
\bibliography{references}

\end{document}

%% file: sec-intro.tex
\section{Introduction}
\noindent
%
Let $ n \in \N $ denote the dimension of the ambient space $ \R^n $.
\citeauthor{Doignon1973}~\cite{Doignon1973} obtained the following analog of the classical theorem of \citeauthor{Helly1923}~\cite{Helly1923}:
If convex sets $C_1,\ldots,C_m$ ($m \in \N$) have no point of $ \Z^n $ in common, then there exists a subset $I$ of $\{1,\ldots,m\}$ with at most $2^n$ elements such that the sets $C_i$ with $i \in I$ already have no point of $ \Z^n $ in common.
It is not possible to replace~$2^n$ by a smaller number. This result was rediscovered independently by \citeauthor{Bell1977}~\cite{Bell1977},
\citeauthor{Scarf1977}~\cite{Scarf1977} and \citeauthor{Hoffman78}~\cite{Hoffman78}. In this paper, we continue the recent studies in \cite{AlievBDL14,ChestnutHZ15,DeLoeraLHRS15} on quantitative versions of Doignon's theorem. Our main object of interest is the following number:

\begin{defn}[Quantitative Helly number] \label{defQuantHelly}
 Let $S \subseteq \R^n$ and $k \in \N_0$. We define the quantitative Helly number $c(S,k)$ as the smallest number $ t \in \N_0$ satisfying the following:
\begin{itemize}
	\item[] If convex sets $C_1,\ldots,C_m$ $(m \in \N)$ have exactly $k$ points of $S$ in common, then there exists a subset $I$ of $\{1,\ldots,m\}$ with at most $t$ elements such that the sets $C_i$ with $i \in I$ already have exactly $k$ points of $S$ in common.
\end{itemize}
In the degenerate case that no such number $t$ exists, let $c(S,k):=\infty$, and if there is no convex set that contains  exactly $k$ points of $S$, let $c(S,k):=-\infty$.
\end{defn}

It turns out that restricting $C_1,\ldots,C_m$ in Definition~\ref{defQuantHelly} to closed halfspaces gives an equivalent definition of $c(S,k)$; see Lemma~\ref{lemRestrictionToClosedHalfSpaces} in Section~\ref{sectPolytopal}.
The number
\[
h(S):=c(S,0)
\]
is called the \emph{Helly number} of (the space) $S$; see \cite{AmentaDS15,Averkov13,AverkovW12}.
\citeauthor{Doignon1973}'s theorem gives the equality $c(\Z^n,0) = h(\Z^n) = 2^n$. Since $c(S,k)$ and by this also $h(S)$ is invariant under non-singular affine transformations, Doignon's theorem can also be formulated in a coordinate-free form in terms of lattices and their rank.
The values of $c(S,k)$ for $S=\R^n$ correspond to the classical theorems of Helly $c(\R^n,0) = h(\R^n) = n+1$
and Steinitz $c(\R^n,1)=2n$  (for the latter, see the explanation given in \cite{Gruenbaum1962}); one obviously has $c(\R^n,k)=-\infty$ for every $k\geq2$. The study of  $c(S,k)$ for $S=\Z^n$ is motivated by applications to integer linear programming and has become a very active research topic;
see  \cite{AlievBDL14,AmentaDS15,Averkov13,AverkovW12,DeLoeraLHORP15,Hoffman78} for more information and \cite{BasuO2015,DeLoeraLHORP15,de2015beyond} for related algorithmic research.

In view of applications, it is interesting to describe the asymptotic behavior of $c(\Z^n,k)$ and to understand how
$c(\Z^n,k)$ can be computed in concrete situations.  
Already in \citeauthor{Bell1977}'s work~\cite{Bell1977} one can find a generalization of the inequality $h(\Z^n) \le 2^n$ which can be formulated as
\begin{equation*} \label{eqBellBound}
    c(\Z^n, k) \le (k+2)^n.
\end{equation*}
\citeauthor{AlievBDL14}~\cite{AlievBDL14} introduced $c(\Z^n,k)$ explicitly and improved Bell's bound to
\begin{equation}
    \label{eqAlievEtAlBound}
    c(\Z^n, k) \le \ceil{2(k+1)/3} (2^n - 2 ) + 2.
\end{equation}
Thus, the growth of $c(\Z^n,k)$ is at most linear in $k$.
Recently \citeauthor{ChestnutHZ15}~\cite{ChestnutHZ15} showed that this number grows only sublinearly in $ k $ by proving
\begin{equation}
	\label{eqChestnutEtAlupper}
    c(\Z^n, k) \le C \cdot k  (\log \log k)  (\log k)^{-1/3} \cdot 2^n
\end{equation}
whenever $\log k > 1$ and $n \in \N$, where $C>0$ is an (unknown) absolute constant.
As a complement to the upper bound, \citeauthor{ChestnutHZ15}~\cite{ChestnutHZ15}  established the lower bound
\begin{equation}
    \label{eqChestnutLowerBound}
    c(\Z^n, k) = \orderOmega(k^{\frac{n-1}{n+1}})
\end{equation}
for every fixed $ n\in\N $ and showed that this lower bound is asymptotically tight for $ n = 2 $.

\subsection*{Our contribution.}
We study $c(S,k)$ in the case of discrete $S$, paying special attention to $S=\Z^n$.
We call a set $ S \subseteq \R^n $ \emph{discrete} if every bounded subset of $ S $ is finite.
Our first main result provides an exact `polytopal description' of $c(S,k)$, which we  use as a tool in the proofs of all the other results. Let $\cP(S)$ be the set of all polytopes whose vertices belong to $S$. In particular, $\cP(\Z^n)$ is the well-known family of \emph{integral polytopes}. For $k \in \N_0$ we introduce
\[
	g(S,k) := \max \setcond{\nverts{P}}{P \in \cP(S), \,\card{\nonverts{P}{S}} = k}.
\]
Here, as usual $\vertices(P)$ denotes the set of all vertices of $P$. Note that in degenerate cases, $g(S,k)$ can be $-\infty$ or $\infty$. It turns out that the sequence $g(S,0), g(S,1), \ldots$ determines the sequence $c(S,0), c(S,1),\ldots $ completely:

\begin{thm} \label{thmPolytopalDescr}
	Let $S \subseteq \R^n$ be discrete and $k \in \N_0$. Then, one has
	\begin{equation} \label{eqPolytopalDescr}
		c(S,k) = \max \setcond{g(S,\ell) + \ell -k}{\ell \in \{0,\ldots,k\}, \ g(S,\ell) + \ell - k \ge 0}
	\end{equation}
	Furthermore, the condition $c(S,k)> -\infty$ is equivalent to $k \le |S|$, and under this condition, $c(S,k)$ can be represented recursively as
	\begin{align} 
		c(S,0) = g(S,0) \quad\text{and}\quad c(S,k) = \max \{ c(S,k-1)-1, g(S,k) \} \quad\text{for} \ 0 < k \le |S|. \label{eqCRecursion}
	\end{align}
\end{thm}

\begin{figure}
	\begin{tikzpicture}[scale=0.45, every node/.style={scale=0.7}]
	\def\radius{0.15}
	\tikzstyle{line} = [thick];
	
	\foreach \dx in {0, 5, 10, 15, 20, 25} {
		\foreach \x\y in {0/0, 0/1, 0/2, 1/0, 1/1, 1/2, 2/0, 2/1, 2/2} {
			\fill (\x + \dx, \y) circle (\radius);
		}
	}
	
	\draw[line] (0,0) -- (1,0) -- (1,1) -- (0,1) -- cycle;
	\draw[line] (5,0) -- (6,0) -- (7,1) -- (7,2) -- (6,2) -- (5,1) -- cycle;
	\draw[line] (10,0) -- (12,0) -- (12,1) -- (11,2) -- (10,1) -- cycle;
	\draw[line] (15,0) -- (17,0) -- (17,1) -- (16,2) -- (15,2) -- cycle;
	\draw[line] (25,0) -- (27,0) -- (27,2) -- (25,2) -- cycle;
	
	\node at ( 1,-0.8) {$ g(S,0) = 4 $};
	\node at ( 6,-0.8) {$ g(S,1) = 6 $};
	\node at (11,-0.8) {$ g(S,2) = 5 $};
	\node at (16,-0.8) {$ g(S,3) = 5 $};
	\node at (21,-0.8) {$ g(S,4) = -\infty $};
	\node at (26,-0.8) {$ g(S,5) = 4 $};
	\end{tikzpicture}
	\caption{Values of $ g(S,k) $ for $ S = \{0,1,2\}^2 $ and $ k \in \{0,\dotsc,5\} $ together with polytopes attaining
		these values. The respective values of $c(S,k)$ can be determined using \eqref{eqCRecursion}; see also Figure~\ref{figureCNumbers}.}
	\label{figGNumbers}
\end{figure}
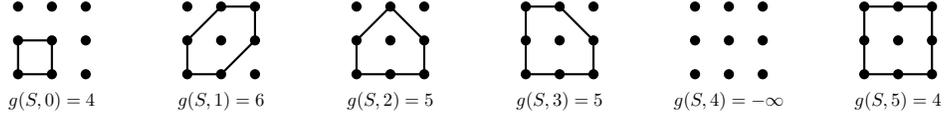

The polytopal representation of the Helly number $h(S) = g(S,0)$ provided in~\eqref{eqCRecursion} was given in  \cite[Lem.~2.2]{DeLoeraLHORP15}. Clearly, \eqref{eqCRecursion} implies $g(S,k) \ge g(S,k-1) -1$ for all $k \in \N$ with $k \le |S|$; this inequality was derived in \cite{ChestnutHZ15} for $S=\Z^n$. 

Theorem~\ref{thmPolytopalDescr} shows that problems for $c(S,k)$ can be reworded as problems for polytopes in $\cP(S)$. One can reformulate \eqref{eqPolytopalDescr} without any use of $g(S,0),\ldots,g(S,k)$ as 
\begin{equation} \label{eqPolytInterp}
	c(S,k) = \max \setcond{\card{S \cap P} -k}{P \in \cP(S), \,\card{\nonverts{P}{S}} \le k \le \card{S \cap P}}.
\end{equation}
Representation~\eqref{eqPolytopalDescr} immediately implies the bounds
\begin{equation}
	\label{boundsCthrougG}
	g(S,k) \le c(S,k) \le \max \{g(S,0),\ldots,g(S,k)\}.
\end{equation}
Thus, one can bound $c(S,k)$ from above by bounding $g(S,0),\ldots,g(S,k)$. Moreover, if $g(S,k)$ turns out to be a largest value among $g(S,0),\ldots,g(S,k)$, one even has $c(S,k)=g(S,k)$. Using the upper bound in \eqref{boundsCthrougG}, we derive the following general upper bound on $c(S,k)$.
\begin{thm}
    \label{thmUpperBoundGeneral}
    Let $S$ be a discrete subset of $\R^n$ with $\card{S} \ge 2$ and let $k \in \N_0$. Then one has
    \[
        c(S,k) \le \floor{(k+1)/2} (h(S)-2) + h(S).
    \]
\end{thm}

Theorem~\ref{thmUpperBoundGeneral} illustrates that the bound $c(S,k)=\orderO(k)$ follows from $h(S)<\infty$.
Highlighting this message is one of the points of motivation for considering general discrete sets $S$. For $S=\Z^n$, Theorem~\ref{thmUpperBoundGeneral} implies \eqref{eqAlievEtAlBound} and improves it by a constant factor for all $n \in \N$ and sufficiently large $k$.

The next two results are for $S=\Z^n$. The first one implies that the exact asymptotic order of $c(\Z^n,k)$ is  $\Theta(k^{\frac{n-1}{n+1}})$ for every fixed $n \in \N$.

\begin{thm}
	\label{thmLowerAndUpperBoundIntegers}
	Let $n,k \in \N$. Then one has
	\begin{equation*} \label{eqMainLowerAndUpperBound}
		\floor{(k / (2n) )^{\frac{1}{n+1}}}^{n-1}
		\le c(\Z^n, k)
		\le (3n)^{5n} \cdotp k^{\frac{n-1}{n+1}}.
	\end{equation*}
\end{thm}

The lower bound of Theorem~\ref{eqMainLowerAndUpperBound} is a concrete version of \eqref{eqChestnutLowerBound}, which we obtain using a short elementary argument. The upper bound can be considered as the main contribution of this paper.

Currently, the values of $c(\Z^n,k)$ are known exactly only in a few cases. \citeauthor{AlievBDL14}~\cite{AlievBDL14} observed that their inequality \eqref{eqAlievEtAlBound} holds with equality for $k=0$ and $k=1$, where the case $k=0$ corresponds to Doignon's theorem.
\citeauthor{ChestnutHZ15}~\cite{ChestnutHZ15} showed that \eqref{eqAlievEtAlBound} still holds with equality for $k=2$.
Continuing this line of research, we determine $c(\Z^n,k)$ for $k=3$ and $k=4$. Thus, the five values of $c(\Z^n,k)$ which are now known exactly are as follows.

\begin{thm}
	\label{thmSpecialValues}
	Let $n\geq2$. Then one has
	\[
	c(\Z^n,0)=2^n,\quad c(\Z^n,1)=c(\Z^n,2)=c(\Z^n,3)=2^{n+1}-2\quad\text{and}\quad c(\Z^n,4)=2^{n+1}.
	\]
\end{thm}

\subsection*{Applications.} Starting from \eqref{eqPolytInterp}, one can use straightforward (but somewhat tedious) arguments for approximating polytopes by strictly convex sets in order to provide the following description of $c(S,k)$:

\begin{cor} \label{corStrictlyConvex}
	Let $S \subseteq \R^n$ be discrete and let $k \in \N_0$. Then $c(S,k)$ is the maximum number
	of points of~$S$ lying in the boundary of a strictly convex body that contains exactly~$k$ points of~$S$ in its
	interior.
\end{cor}

In view of this corollary, the equality $c(\Z^n,1) = 2^{n+1} - 2$ from Theorem~\ref{thmSpecialValues} implies an old result of Minkowski \cite[Thm.~30.2]{Gruber2007} saying that every $0$-symmetric strictly convex body with exactly one interior integer point contains at most $2^{n+1} - 1$ integer points in total. (Here the strict convexity is a crucial assumption; see \cite{GonzalezHenze14} for an analogous investigation without this assumption.) Thus, Corollary~\ref{corStrictlyConvex} shows that the study of $c(\Z^n,k)$ can be viewed as a research in geometry of numbers dealing with the case of strictly convex bodies.


Yet another interpretation of $c(S,k)$ provides a link to the cutting plane theory for mixed-integer optimization problems:
\begin{thm} \label{thmMaximalSets}
    Let $S \subseteq \R^n$ be discrete, let $k \in \N_0$ and $c(S,k)<\infty$. Then $c(S,k)$ is the maximum number
    of facets of an $n$-dimensional polyhedron that contains exactly~$k$ points of~$S$ in its interior and is
    inclusion-maximal with respect to this property.
\end{thm}
A discussion of the connections with the cutting plane theory and inclusion-maximal convex sets with $k$ interior points in $S$ is postponed to Appendix~\ref{sectMaximalFree}.

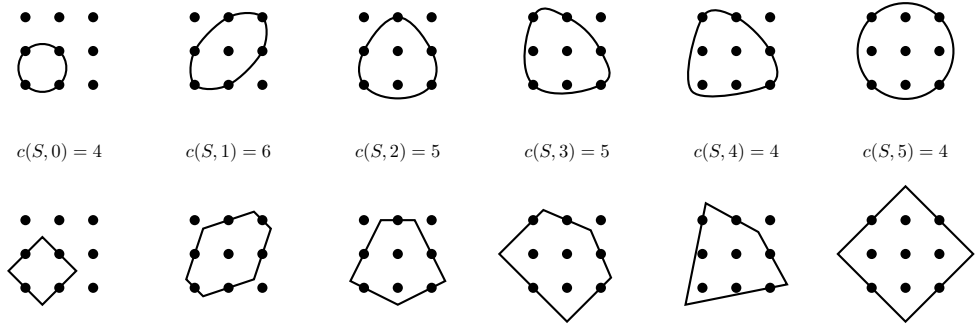
\begin{figure}
	\begin{tikzpicture}[scale=0.45, every node/.style={scale=0.7}]
	\def\radius{0.15}
	\tikzstyle{line} = [thick];
	
	
	\foreach \dx in {0, 5, 10, 15, 20, 25} {
		\foreach \x\y in {0/0, 0/1, 0/2, 1/0, 1/1, 1/2, 2/0, 2/1, 2/2} {
			\fill (\x + \dx, \y) circle (\radius);
		}
	}
	
	\draw[line] plot [smooth cycle, tension=1] coordinates {(0,0) (1,0) (1,1) (0,1)};
	\draw[line] plot [smooth cycle, tension=0.6] coordinates {(5,0) (6,0) (7,1) (7,2) (6,2) (5,1)};
	\draw[line] plot [smooth cycle, tension=0.7] coordinates {(10,0) (11,-0.4) (12,0) (12,1) (11,2) (10,1)};
	\draw[line] plot [smooth cycle, tension=0.8] coordinates {(15,0) (17,0) (17,1) (16,2) (15,2)};
	\draw[line] plot [smooth cycle, tension=0.6] coordinates {(19.7,-0.2) (22,0) (22,1) (21,2) (20,2)};
	\draw[line] plot [smooth cycle, tension=1] coordinates {(25,0) (27,0) (27,2) (25,2)};
	
	\node at ( 1,-2) {$ c(S,0) = 4 $};
	\node at ( 6,-2) {$ c(S,1) = 6 $};
	\node at (11,-2) {$ c(S,2) = 5 $};
	\node at (16,-2) {$ c(S,3) = 5 $};
	\node at (21,-2) {$ c(S,4) = 4 $};
	\node at (26,-2) {$ c(S,5) = 4 $};
	
	
	\foreach \dx in {0, 5, 10, 15, 20, 25} {
		\foreach \x\y in {0/0, 0/1, 0/2, 1/0, 1/1, 1/2, 2/0, 2/1, 2/2} {
			\fill (\x + \dx, \y - 6) circle (\radius);
		}
	}
	
    \draw[line] (0.5,-4.5) -- (1.5,-5.5) -- (0.5,-6.5) -- (-0.5,-5.5) -- cycle;
    \draw[line] (4.75,-5.75) -- (5.25,-6.25) -- (6.75,-5.75) -- (7.25,-4.25) -- (6.75,-3.75) -- (5.25,-4.25) -- cycle;
    \draw[line] (9.6,-5.8) -- (11,-6.5) -- (12.4,-5.8) -- (11.5,-4) -- (10.5,-4) -- cycle;
    \draw[line] (14,-5) -- (16,-7) -- (17.3,-5.7) -- (16.7,-4.3) -- (15.3,-3.7) -- cycle;
    \draw[line] (19.5,-6.5) -- (22.5,-5.9) -- (21.65,-4.35) -- (20.1,-3.5) -- cycle;
    \draw[line] (26,-3) -- (28,-5) -- (26,-7) -- (24,-5) -- cycle;
	\end{tikzpicture}
	\caption{Values of $ c(S,k) $ for $ S = \{0,1,2\}^2 $ and $ k \in \{0,\dotsc,5\} $ together with strictly convex
	bodies (see Corollary~\ref{corStrictlyConvex}) and inclusion-maximal polytopes (see Theorem~\ref{thmMaximalSets})
	attaining these values.}
	\label{figureCNumbers}
\end{figure}

Clearly, the number $\max \{c(S,0),\ldots,c(S,k)\}$ can be defined by replacing `exactly $k$ points' with `at most $k$ points' in Definition~\ref{defQuantHelly}. This number has recently been introduced in  \cite[Def.~1.8]{DeLoeraLHRS15} and coincides with $\h_S(k+1)$ in the notation of \cite{DeLoeraLHRS15}. The inequalities in~\eqref{boundsCthrougG} imply that $\h_S(k+1)$ can be described as
\[
	\h_S(k+1) = \max \{c(S,0),\ldots,c(S,k)\} = \max \{ g(S,0),\ldots,g(S,k)\}.
\]
That is, $\h_S(k+1)$ is the maximum of $\nverts{P}$ taken over polytopes $P \in \cP(S)$ satisfying $\card{\nonverts{P}{S}} \le k$. \citeauthor{DeLoeraLHRS15}~\cite{DeLoeraLHRS15} also introduced the so-called quantitative $S$-Tverberg
number $\T_S(m,k)$ and provided the bound  $\T_S(m,k)\le \h_S(k)(m-1)kn+k$ on $\T_S(m,k)$ in terms of $\h_S(k)$. Combining the latter bound with our upper bounds on $c(S,k)$ we immediately obtain upper bounds on $\T_S(m,k)$ for general discrete sets $S$ as well as $S=\Z^n$.

\subsection*{Open questions} We formulate several questions that arise  naturally.

\begin{enumerate}[label=\arabic*.]
	\item It is not clear whether sublinearity of $c(\Z^n,k)$ in $k$ relies on specific properties of~$\Z^n$ or can be derived directly from the fact $h(\Z^n)<\infty$. This leads to the following question. If $S \subseteq \R^n$ is an arbitrary discrete set with $h(S)< \infty$, can the linear bound $c(S,k) = \orderO(k)$ be improved to a sublinear one?
	\item One has $c(\Z^n,k) = \orderO(2^n)$ for every fixed $k \in \N_0$ and $c(\Z^n,k) = \orderO(k^{\frac{n-1}{n+1}})$ for every fixed $n \in \N$.  What is the behavior of $c(\Z^n,k)$ when both $k$ and $n$ vary? Does one have $c(\Z^n,k) \le C \cdot 2^n \cdot k^{\frac{n-1}{n+1}}$ for all $k,n \in \N$, where $C>0$ is some absolute constant?
	\item The example $S=\{0,1,2\}^2$ shows that $c(S,k)$ and $g(S,k)$ are not
	the same in general; see Figures~\ref{figGNumbers} and \ref{figureCNumbers}. It is however not clear under which conditions on $S$ one has the equality $c(S,k)=g(S,k)$ for every $k \in \N_0$. In particular, for $S=\Z^n$, the following question is open. Do there exist $n \in \N$ and $k \in \N_0$ with $c(\Z^n,k) > g(\Z^n,k)$?
	Note that, in view of Theorem~\ref{thmSpecialValues}, one has $c(\Z^n,k) =g(\Z^n,k)$ for $n \ge 2$ and $k \le 4$.
	
	If for some $n \in \N$ there exists $k \in \N$ with $c(\Z^n,k) > g(\Z^n,k)$, then for the smallest such value $k=k'$ one has $c(\Z^n,k') > g(\Z^n,k')$ and $c(\Z^n,k'-1) = g(\Z^n,k'-1)$ so that \eqref{eqCRecursion} yields $g(\Z^n,k') < g(\Z^n,k'-1) - 1$. This means that passing from $k=k'-1$ to $k=k'$ the value $g(\Z^n,k)$ decreases by at least two. Thus, the question can be reformulated as follows: Do there exist $n \in \N$ and $k \in \N$ such that $g(\Z^n,k) \le g(\Z^n,k-1) - 2$?
	
	Relying on Castryck's database~\cite[Rem.~(3.4)]{Castryck2012} of all integral polygons with at most $ 30 $ interior integer points, we determine $g(\Z^2,k)$ for
	all $k \le 30$ in a straightforward way; see Figure~\ref{figPlanarValues}. Our computation verified that $g(\Z^2,k) \ge g(\Z^2,k-1)-1$ holds for every $k \in \{1,\ldots,30\}$. This implies that $g(\Z^2,k)$ and $c(\Z^2,k)$ coincide for $k \le 30$. The computation also determined that, in the given range, $g(\Z^2,k)$ is not monotonic, since the equality $g(\Z^2,k) = g(\Z^2,k-1) - 1$ is attained for $k \in \{5,8,11,22,28\}$ (for the case $k=5$ see also \cite{ChestnutHZ15}).
\end{enumerate}

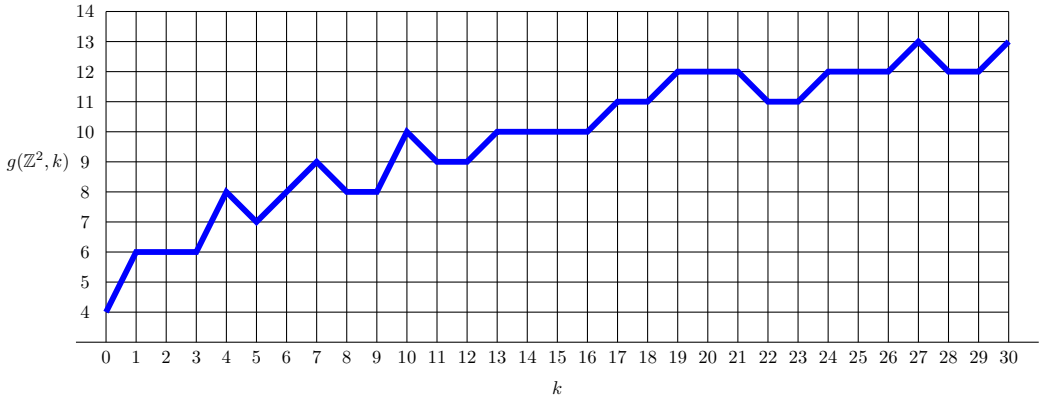
\begin{figure}
    \begin{tikzpicture}[scale=0.4, every node/.style={scale=0.7}]
        \tikzstyle{line} = [thick];
        \def\transp{70};

        \node at (15,1.5) {$ k $};
        \node at (-2.25,9) {$ g(\Z^2,k) $};
        \foreach \x in {0,...,30}
        {
            \draw[thin,black] (\x,3) -- (\x,14);
            \node at (\x,2.5) {$ \x $};
        }
        \foreach \y in {4,...,14}
        {
            \draw[thin,black] (0,\y) -- (30,\y);
            \node at (-0.7,\y) {$  \y $};
        }

        \draw[blue!100,line width=2.2pt] (0,4) -- (1,6) -- (2,6) -- (3,6) -- (4,8) -- (5,7) -- (6,8) -- (7,9) -- (8,8) -- (9,8) -- (10,10) --
            (11,9) -- (12,9) -- (13,10) -- (14,10) -- (15,10) -- (16,10) -- (17,11) -- (18,11) -- (19,12) -- (20,12) --
            (21,12) -- (22,11) -- (23,11) -- (24,12) -- (25,12) -- (26,12) -- (27,13) -- (28,12) -- (29,12) -- (30,13);

        \draw (-1,3) -- (31,3);
    \end{tikzpicture}
    \caption{The values of $ g(\Z^2,k) = c(\Z^2,k) $ for all $ k \in \{0,\dotsc,30\} $.}
    \label{figPlanarValues}
\end{figure}


\subsection*{Structure, notation and terminology} In Sections~\ref{sectPolytopal}--\ref{sectSpecificValues} we present proofs of Theorems~\ref{thmPolytopalDescr}--\ref{thmSpecialValues}, respectively. Appendix~\ref{sectMaximalFree} presents material related to Theorem~\ref{thmMaximalSets}, while in  Appendix~\ref{secAppendixA} we determine an explicit constant in a result of Andrews \cite{Andrews63} that is used in the proof of Theorem~\ref{thmLowerAndUpperBoundIntegers}.

We use the notation $\N_0 := \{0,1,2,\ldots\}$ and $\N:=\{1,2,3,\ldots\}$. Let $[m]:=\{1,\ldots,m\}$ for $m \in \N$ and $[0]:=\emptyset$. For a
set $S$, we denote by $\card{S}$ its cardinality.

We use basic background from the theory of polytopes, convex sets and geometry of numbers; see, for example, \cite{Barvinok02,Gruber2007,GruberL1987,Schneider14}. Throughout, $ n \in \N $ stands for the dimension of the ambient space $\R^n$, which is equipped with the standard inner product $\sprod{\cdot}{\cdot}$.
For a polyhedron $P\subseteq\R^n$, we denote by $\vertices(P)$ and $\facets(P)$
the set of vertices and facets of $P$, respectively, where a \emph{facet} is a non-empty face of $P$ of dimension $\dim(P)-1$. Note that $\facets(P) = \emptyset$ if $P$ consists of one point only. For $ S \subseteq \R^n $, we denote by $\cP(S)$ the family of all polytopes $P$ with $\vertices(P) \subseteq S$. We denote by
$\interior(C)$, $\bd(C)$, $\aff(C)$, and $\conv(C)$ the interior,  the boundary, the affine hull, and the convex hull of $C \subseteq \R^n$, respectively.
We say that $K\subseteq\R^n$ is a \emph{convex body} if $K$ is an $n$-dimensional compact convex set. By $\vol(K)$ we denote the volume of a convex body $K$.

%% file: sec-polytopal.tex
\section{Polytopal interpretation of \texorpdfstring{$c(S,k)$}{c(S,k)}}
\label{sectPolytopal}

In this section, we prove Theorem~\ref{thmPolytopalDescr}.
In our argumentation, we will make use of the fact that one can define $ c(S,k) $ equivalently by restricting $
C_1,\dotsc,C_m $ in Definition~\ref{defQuantHelly} to closed halfspaces.
%

\newcommand{\cH}{\mathcal{H}}
\begin{lem}
    \label{lemRestrictionToClosedHalfSpaces}
    Let $ S $ be a subset of $ \R^n $ (not necessarily a discrete one), let $ k \in \N_0 $ and $t \in \N_0$. Then the following two conditions are equivalent:
    \begin{itemize}
        \item[(i)] If convex sets $C_1,\ldots,C_m$ $(m \in \N)$ have exactly $k$ points of $S$ in common, then there
        exists some $ I \subseteq [m] $ with $ \card{I} \le t $ such that the sets $C_i$ with $i \in I$ already have exactly~$k$
        points of $S$ in common.
        \item[(ii)] If closed halfspaces $H_1,\ldots,H_m$ $(m \in \N)$ have exactly $k$ points of $S$ in common, then
        there exists some $ I \subseteq [m] $ with $ \card{I} \le t $ such that the sets $H_i$ with $i \in I$ already have
        exactly $k$ points of $S$ in common.
    \end{itemize}
\end{lem}
\begin{proof}
	In our proof we borrow ideas from \cite{AverkovW12}. As closed halfspaces are special convex sets, it is clear that~(i) implies~(ii). We verify the converse. Assume that (ii) is fulfilled. Consider convex sets $C_1,\ldots,C_m$ as in (i) such that the set
    \[
	    X := \bigcap_{i =1}^m C_i \cap S
    \]
    has cardinality $k$. We first verify (i) in the case that $C_1,\ldots,C_m$ are polyhedra. In this case each $C_i$ with $i \in [m]$ is the intersection of a finite family $\cH_i$ of closed halfspaces. By construction, $\bigcap_{H \in \cH} H \cap S = X$ for $\cH:=\cH_1 \cup \ldots \cup \cH_m$. Applying (ii) to the family $\cH$ we deduce the existence of $\cH' \subseteq \cH$ with $\card{\cH'} \le t$ such that $\bigcap_{H' \in  \cH'} H' \cap S = X$. Since $\card{\cH'} \le t$ and each $H' \in \cH'$ contains some $C_i$ with $i \in [m]$, there exists an index set $I \subseteq [m]$ with $\card{I} \le t$ such that $\bigcap_{i \in I} C_i \cap S \subseteq \bigcap_{H' \in \cH'} H' \cap S$.  The left-hand side of the latter inclusion contains $X$, while the right-hand side coincides with $X$. Consequently, $\bigcap_{i \in I} C_i \cap S = X$.

    It remains to verify (i) for general convex sets $C_1,\ldots,C_m$. Let $I \subseteq [m]$ be an inclusion-minimal set with $\bigcap_{i \in I} C_i \cap S = X$. We need to verify $\card{I} \le t$. The minimality of $I$ implies that for every $i \in I$ there exists $s_i \in S$ that belongs to $C_j$ with $j \in I \setminus \{i\}$ but does not belong to $C_i$. Consider the polytopes $C_i':= \conv(X \cup \setcond{s_j}{j \in I \setminus \{i\}})$ for $i\in I$. By construction,  the set $\bigcap_{i \in I} C_i' \cap S$ coincides with $X$, while for every proper subset $J$ of $I$, the  set $\bigcap_{j \in J} C_i' \cap S$ is strictly larger than $X$. Since (i) has already been verified for polyhedra, it follows that $\card{I} \le t$.
\end{proof}
\begin{proof}[Proof of Theorem~\ref{thmPolytopalDescr}]
	For verifying \eqref{eqPolytopalDescr}, it suffices to verify \eqref{eqPolytInterp}, as both representations are equivalent. We denote by $m$ the right-hand side of \eqref{eqPolytInterp}.
	
	We first show $c(S,k) \ge m$. Consider an arbitrary polytope $P$ as in \eqref{eqPolytInterp} and let  $t = |S \cap P| -k$. We need to show $c(S,k) \ge t$. One has $t \ge 0$ and the polytope $P$ has at least~$t$ vertices. We choose $t$ pairwise distinct vertices $v_i$ with $i \in [t]$ of $P$ and introduce the sets $C_i :=\conv(S \cap P \setminus \{v_i\})$. By construction, the set $X:=\bigcap_{i=1}^t C_i \cap S = S \cap P \setminus \{v_1,\ldots,v_t\}$ has cardinality $t$. On the other hand, for every proper subset $I$ of $[t]$, the set $\bigcap_{i \in I} C_i \cap S$ strictly contains $X$ as it additionally contains points $v_i$ with $i \in [t] \setminus I$. This shows $c(S,k) \ge t$. 

    Next, we show $ c(S,k) \le m $. We distinguish the cases
    $c(S,k)=-\infty$, $c(S,k)=0$, $c(S,k)\in\N$, and $c(S,k)=\infty$.
    
    If $c(S,k)=-\infty$ there is nothing to prove. The case $c(S,k)=0$
    readily implies $\card{S}=k$ (note that the intersection of an empty family of convex subsets of $\R^n$ is the whole space~$\R^n$). In this case, The polytope $P=\conv(S)$ satisfies $\card{S\cap P\setminus \vertices(P)} \le k=\card{S\cap P}$.
    Therefore, $m\geq0$.
	
		Now, we consider the case $c(S,k) \in \N$ and let $t=c(S,k)$. By Lemma~\ref{lemRestrictionToClosedHalfSpaces}, there exist $a_1,\ldots,a_t \in \R^n$ and $\beta_1,\ldots,\beta_t \in \R$ such that the system
		\begin{equation} \label{eqOriginalSystem}
			\sprod{a_1}{x} \le \beta_1,\ldots,\sprod{a_t}{x} \le \beta_t
		\end{equation}
		of $t$ inequalities has exactly $k$ solutions in $S$ and such that, for every $i \in [t]$, there exists a point $s_i \in S$ that satisfies all but the $i$-th inequality of \eqref{eqOriginalSystem}. Let $X$ be the set of  solutions of~\eqref{eqOriginalSystem} that lie in $S$. We introduce the finite subset $S':= \conv(X \cup \{s_1,\ldots,s_t\}) \cap S$ of~$S$ and consider $t$ parameters $\gamma_1,\ldots,\gamma_t \in \R$. We first choose  $\gamma_i$ to be slightly larger than~$\beta_i$ for every $i \in [t]$. Since $S'$ is finite, with this choice, the system
		\begin{equation}
		\label{eqNewSystem}
			\sprod{a_1}{x} < \gamma_1,\ldots,\sprod{a_t}{x} < \gamma_t
		\end{equation}
		of $t$ strict inequalities has the same set $X$ of solutions in $S'$ as the system \eqref{eqOriginalSystem}. Furthermore, for each $i \in [t],$ the point $s_i \in S'$ satisfies all but the $i$-th inequality of \eqref{eqNewSystem}. Now, for each $i \in [t]$, we can enlarge $\gamma_i$ until we reach a value for which there exists a point $s_i' \in S' \setminus X$ that satisfies all but the $i$-th inequality of \eqref{eqNewSystem} and also satisfies the equality $\sprod{a_i}{s_i'} = \gamma_i$. Adjusting the values of $\gamma_1,\ldots,\gamma_t$ consecutively using the above procedure, we ensure that for the points $s_1',\ldots,s_t' \in S'$ one has $\sprod{a_i}{s_j'} < \gamma_i$ and $\sprod{a_i}{s_i'} = \gamma_i$ for all $i, j \in [t]$ with $i \ne j$. Thus the points $s_1',\ldots,s_t'$ are $t$ distinct vertices of the polytope $P=\conv(X \cup \{s_1',\ldots,s_t'\})$ and $P$ contains exactly $\card{X} + t=k+t$ points of $S$ in total. Thus, $P$ occurs in the set on the right-hand side of \eqref{eqPolytInterp} and for $P$ one has $|S \cap P| - k = t$. Consequently, $c(S,k) = t \le m$. 
		
		In the case $c(S,k) = \infty$, no choice of $t \in \N$ fulfills the property in Definition~ \ref{defQuantHelly}. Thus, we can apply the arguments of the previous case for every $t \in \N$. This yields $t \le m$ for every $t \in \N$. Hence $m=\infty$. 
		
		Above, we have verified the validity \eqref{eqPolytInterp} (and by this, also \eqref{eqPolytopalDescr}).
		
		We verify the equivalence of $c(S,k) > -\infty$ and $k \le |S|$. One obviously has $c(S,k) = -\infty$ if $k > |S|$. If $k \le |S|$ we can choose $k$ distinct points $s_1,\ldots,s_k \in S$. Then  $S' = \conv(\{s_1,\ldots,s_k\}) \cap S$ is a finite set consisting of at least $k$ points. Clearly, there exists a hyperplane separating exactly $k$ of these points from the remaining ones. This yields the existence of a convex set with exactly $k$ points in $S$ and concludes the proof of the equivalence. 
		
		It remains to verify the recursive representation \eqref{eqCRecursion} under the condition $k \le |S|$. The equality $c(S,0) = g(S,0)$ is a special case of \eqref{eqPolytopalDescr}. Let $0 < k \le |S|$. One has $c(S,k) > -\infty$, which implies $c(S,k) \ge 0$. Thus, the maximum in \eqref{eqPolytopalDescr} is non-negative so that one can omit the condition $g(S,\ell) + \ell - k \ge 0$ and represent $c(S,k)$ as the maximum of $g(S,\ell) + \ell - k$ taken over  $\ell \in \{0,\ldots,k\}$.  Analogously, $c(S,k-1)$ is also non-negative and can be given as the maximum of $g(S,\ell) + \ell - (k-1)$ taken over $\ell \in \{0,\ldots,k-1\}$. From these representations of $c(S,k)$ and $c(S,k-1)$, \eqref{eqCRecursion} follows immediately. 
\end{proof}

%% file: sec-upper-bound-general.tex
\section{An upper bound for general discrete sets}
\label{sectUpperBoundGeneral}
\noindent
In this section, we prove Theorem~\ref{thmUpperBoundGeneral}.
To illustrate our  approach for deriving the upper bound in Theorem~\ref{thmUpperBoundGeneral}, we first present a proof of the weaker bound
\begin{align}
c(S,k) &\le (k+1) h(S) \label{eqnWeakgBound}
\end{align}
for every $k \in \N$ and every non-empty discrete subset $S$ of $\R^n$. Since the right-hand side of \eqref{eqnWeakgBound} is non-decreasing in $k$ and $c(S,k)$ can be bounded in terms of $g(S,0),\ldots,g(S,k)$, it suffices to prove the bound $g(S,k) \le (k+1) h(S)$. Consider an arbitrary $P \in \cP(S)$  with $X =S \cap P \setminus \vertices(P)$ satisfying $\card{X}=k$. Let $X = \{s_1,\ldots,s_k\}$. For a vector $u \in \R^n \setminus \{0\}$ and $i \in [k]$, we denote by $H_i$ the hyperplane orthogonal to $u$ and passing through $s_i$. If $u \in \R^n \setminus \{0\}$ is chosen generically, the hyperplanes $H_1,\ldots,H_k$ are pairwise distinct and one has $H_i \cap \vertices(P) = \emptyset$ for every $i \in [k]$. The hyperplanes $H_1,\ldots,H_k$ decompose $\R^n$ into $k+1$ polyhedral regions: Two closed halfspaces and $k-1$ slabs (where a slab is the convex hull of two distinct parallel hyperplanes). We index these regions by $1,\ldots,k+1$. For $i \in [k+1]$, by $P_i$ we denote the convex hull of those points of $S \cap P$ that are contained in the $i$-th region. By construction $S \cap P_i = \vertices(P_i)$ and $\vertices(P) \subseteq \bigcup_{i=1}^n 
\vertices(P)$. Consequently, 
\[
	\nverts{P} \le \sum_{i=1}^{k+1} \,\nverts{P_i} \le (k+1) h(S).
\]

The rest of this section is devoted to proving Theorem~\ref{thmUpperBoundGeneral} which improves the bound~\eqref{eqnWeakgBound} by roughly a factor of~$1/2$.
In Lemma~\ref{lemUpperBoundGeneralDataQ}, given below, we estimate the number of vertices of a polytope $P\in\cP(S)$ using the combinatorics of the polytope $Q:=\conv(S \cap P \setminus \vertices(P))$.  We apply Lemma~\ref{lemUpperBoundGeneralDataQ} in the cases $0 \le \dim(Q) \le 2$ for proving  Theorem~\ref{thmUpperBoundGeneral} in this section and Theorem~\ref{thmSpecialValues} in Section~\ref{sectSpecificValues}.

\begin{lem} \label{lemUpperBoundGeneralDataQ}
	Let $S \subseteq \R^n$ be discrete. Let $P\in\cP(S)$ be such that $Q:=\conv(\nonverts{P}{S})$ is non-empty and let $S':= \aff(Q) \cap S$. Then
	\[
		\nverts{P} \le 2 \left( h(S) - h(S') \right) + \card{\facets(Q)} \left( h(S') - \dim(Q) \right).
	\]
	In particular, if $\dim(Q) \le 1$, one has 
	\begin{equation}
		\label{eqBoundWhenDimQ01}
		\nverts{P} \le 2 h(S) - 2.
	\end{equation}
\end{lem}
\begin{proof}
We distinguish several cases according to the dimension of $Q$. Consider the case $\dim(Q)=0$, that is, $Q$ consists of one point.
For a generically chosen hyperplane $H$ that contains $Q$ one has $H \cap \vertices(P) = \emptyset$.
The hyperplane $H$ splits $\vertices(P)$ into two disjoint subsets $V_1$ and $V_2$ lying on different sides of $H$.
The polytopes $P_{\ell}:=\conv(V_\ell \cup Q)$ with $\ell \in \{1,2\}$ satisfy $S \cap P_\ell = \vertices(P_\ell) = V_\ell \cup Q$.
Consequently,
\[
\nverts{P} = \card{V_1} + \card{V_2}  = \nverts{P_1} + \nverts{P_2} - 2 \le 2 h(S)-2.
\]

We switch to the case $1\leq\dim(Q)\leq n-1$. Let $t:=\nverts{Q}$, $s:=\card{\facets(Q)}$ and $V':=\vertices(P) \cap \aff(Q)$.
	Let $F_1,\ldots,F_s$ be all facets of $Q$.  
	Consider a point~$a$ in the relative interior of $Q$. The affine space $\aff(Q)$ can be subdivided into $s$ polyhedral cones, where for each $i \in [s]$, the $i$-th cone is defined as the union of all rays emanating from $a$ and passing through points of $F_i$. Choosing the point $a$ in the relative interior of $Q$ generically, we ensure that none of the above cones contains points of $V'$ in their relative boundary.  
		
	In the $i$-th cone we pick a subset $X_i$ of~$F_i \cap S$ of cardinality $\dim(Q)$ satisfying $\conv(X_i)\cap S=X_i$. Let $V_i'$ be the set of all points of $V'$ contained in the $i$-th cone and 
	let 
	\[
		P_i':=\conv(V_i' \cup X_i).
	\]
	By construction,~$P_i'$ satisfies $P_i' \cap S' = \vertices(P_i') = V_i' \cup X_i.$
	We also introduce 
	\[m:= \max \setcond{\nverts{P_i'}}{i \in [s]}.\] Without loss of generality, let $m=\nverts{P_1'}$.
	
	Consider a hyperplane $H$ that contains $\aff(Q)$. Choosing such a hyperplane  $H$ generically, we ensure $H \cap \vertices(P) = V'$. The hyperplane $H$ splits $\vertices(P) \setminus V'$ into two disjoint subsets, say $V_1$ and $V_2$.
	For each $\ell \in \{1,2\}$, the polytope 
	\[
		P_\ell := \conv( V_\ell \cup \vertices(P_1'))
	\] has the property $S \cap P_\ell = \vertices(P_\ell) = V_\ell \cup \vertices(P_\ell')$.
	Taking into account $\nverts{P_\ell} \le h(S)$, for $\ell \in\{1,2\}$, and $m \le h(S')$, and $s\geq2$, we obtain the desired bound:
	\begin{align*}
		\nverts{P}
		& = \card{V_1} + \card{V_2} + \sum_{i=1}^s \card{V_i'}
		\\ & \le (\nverts{P_1} - m) +  (\nverts{P_2} -m) + s (m-\dim(Q))
		\\ & = \nverts{P_1} + \nverts{P_2} + (s-2) m - s \dim(Q)
		\\ & \le 2 h(S) + (s-2) h(S') - s\dim(Q)
		\\ & = 2 (h(S) -h (S')) + \card{\facets(Q)}\left(h(S')-\dim(Q)\right).
	\end{align*}
The case $\dim(Q)=n$ is analogous to the previous one with the exception that no separating hyperplane $H$ needs to be introduced. 
In the notation from the previous case, we have $V'=\vertices(P)$ and obtain 
\[
 \nverts{P} = \sum_{i=1}^s \card{V_i'} \leq s(m-n) \leq \card{\facets(Q)}\left(h(S)-n\right).\qedhere
\]
\end{proof}

\begin{proof}[Proof of Theorem~\ref{thmUpperBoundGeneral}]
In view of \eqref{boundsCthrougG}, it suffices to verify the asserted inequality for $g(S,k)$ in place of $c(S,k)$.
    We assume $ n \ge 2 $ as for $ n = 1 $, one has $ g(S,k) = h(S) = 2 $ and the assertion is trivial.
    %
    We proceed by induction on $k$. The case $ k = 0 $ corresponds to the inequality $ g(S,0) \le h(S) $, which holds (with equality) in view of Theorem~\ref{thmPolytopalDescr}.
	It follows directly from Lemma~\ref{lemUpperBoundGeneralDataQ}, that $ g(S,k) \le 2 h(S) - 2 $, for $ k \in \{1,2\} $. This inequality is precisely what we need to prove in these cases.
    %

    Next, let $k \ge 3$ and assume that $ g(S,k')
    \le \floor{(k'+1)/2} (h(S)-2) + h(S) $ has been verified for every $ k' \in \{0,\ldots,k-1\}$.
    Consider an arbitrary polytope $ P $ with $ \vx(P) \subseteq S $ and such that the set $ X:= \nonverts{P}{S} $ consists of exactly $k$ points.
    Since $k \ge 3$, the polytope $\conv(X)$ has dimension at least one.

    We pick an edge $E$ of $\conv(X)$ (where $E= \conv(X)$ if $\conv(X)$ is one-dimensional) and two consecutive points $s_1$ and $s_2$ of $S$ lying in $E$. Since $E$ is one-dimensional, the set 
    \[
	    V' := \aff(E) \cap \vertices(P)
    \] consists of at most two points. A generically chosen hyperplane $H$ with $H \cap \conv(X) = E$ satisfies $H \cap \vertices(P) = V'$. The hyperplane $H$ splits $\vertices(P) \setminus V'$ into two sets, say $V_1$ and~$V_2$, where we assume that~$V_2$ lies on the same side of $H$ as $\conv(X)$.
    By construction, 
    \[
      P_1 := \conv(V_1 \cup \{s_1,s_2\})
    \] 
    is a polytope with $S \cap P_1 = \vertices(P_1) = V_1 \cup \{s_1,s_2\}$. Furthermore, 
    \[
	    P_2 := \conv(V_2 \cup \{s_1,s_2\})
    \] is a polytope with $V_2 \cup \{s_1,s_2\} = \vertices(P_2)$. One has $S \cap P_2 \setminus \vertices(P_2) \subseteq X \setminus \{s_1,s_2\}$ and so $k' := \card{\nonverts{P_2}{S}} \le k-2$.

    Summarizing, we obtain the desired bound on $\nverts{P}$ using the induction assumption:
    \begin{align*}
    \nverts{P} &  = \card{V_1} + \card{V_2} + \card{V'}
    \\ & \le (\nverts{P_1}-2) + (\nverts{P_2}-2) + 2
    \\ & = \nverts{P_1} + \nverts{P_2} -2
    \\ & \le h(S) + \floor{(k'+1)/2} (h(S)-2) + h(S) -2
    \\ & = \left(\floor{(k'+1)/2} + 1 \right) (h(S) - 2) + h(S)
    \\ & \le \floor{(k+1)/2} (h(S)-2) + h(S).\qedhere
    \end{align*}
\end{proof}

%% file: sec-bounds-integers.tex
\section{Bounds for the integer lattice}
\label{sectBoundsIntegers}
\noindent
In this section, we prove Theorem~\ref{thmLowerAndUpperBoundIntegers}.
%
%
\noindent
We first give a proof for the upper bound.
To this end, let us recall three results in geometry of numbers.
The first one relates the number of vertices of an integral polytope to its volume.
\begin{thm}[\citeauthor{Andrews63}~\cite{Andrews63}]
    \label{thmAndrews}
    For every $ n \in \N $ there exists a constant $ \alpha(n) $ such that for every $ n $-dimensional polytope $ P
    \in \cP(\Z^n) $ one has
    \[
        \nverts{P} \le \alpha(n) \cdotp \vol(P)^{\frac{n-1}{n+1}}.
    \]
    The constant can be chosen as $ \alpha(n) = (3n)^{4n} $.
\end{thm}
\noindent
We also refer to \cite{Barany08} for a discussion of various proofs of Theorem~\ref{thmAndrews} available in the literature. A discussion of the choice of $\alpha(n)$ is given in Appendix~\ref{secAppendixA}.
Second, we need an upper bound on the volume of a convex body in terms of the so-called coefficient of asymmetry.
Given a convex body $ K \subseteq \R^n$ and $ x \in \interior(K) $, the \emph{coefficient of asymmetry} of $ K $ with respect to $ x $
can be defined as
\begin{equation*}
    \ac(K,x) := \min \setcond{ \lambda \ge 0 }{ x - K \subseteq \lambda(K - x) },
\end{equation*}
see \cite{Gruenbaum63}. Using the cancelation law for Minkowski addition of convex bodies (see \cite[p.~48]{Schneider14}), the latter can be reformulated as
\begin{equation*}
   \ac(K,x) =  \min \setcond{ \lambda \ge 0 }{ x + \frac{1}{\lambda + 1} (K - K) \subseteq K }.
\end{equation*}

\noindent \citeauthor{lagariasziegler1991bounds}~\cite{lagariasziegler1991bounds} observe that the following result can be derived from an inequality of van der Corput; see also \cite[p.~47 \& p.~127]{GruberL1987}.
\begin{thm}[\citeauthor{lagariasziegler1991bounds}~\cite{lagariasziegler1991bounds}]
    \label{thmMahler}
    Let $ K \subseteq \R^n $ be a convex body and let $ x \in \Z^n \cap \interior(K) $. Then, one has
    \[
        \vol(K) \le \left(1 + \ac(K,x)\right)^n \cdotp |\Z^n \cap \interior(K)|.
    \]
\end{thm}
\noindent
Finally, we will make use of the following version of flatness theorems:
\begin{thm}[{\cite[Thm.~8.3]{Barvinok02}}]
    \label{thmFlatness}
    For every $ n \in \N $ there exists a constant $ \flt(n) $ such that for every convex body $ K \subseteq \R^n $ with
    $ K \cap \Z^n = \emptyset $ there exists a vector $ u \in \Z^n \setminus\{0\}$ with
    \[
        w(K, u) := \max_{x \in K} \sprod{x}{u}  - \min_{x \in K} \sprod{x}{u} \le \flt(n).
    \]
    The constant can be chosen as $ \flt(n) = n^{5/2} $.
\end{thm}

While flatness theorems providing better choices of $\flt(n)$ in the sense of asymptotic behavior are available in the literature, we prefer the bound $n^{5/2}$, because this bound is explicit and our subsequent estimates are not very sensitive with respect to the exponent $5/2$. The exponent $5/2$ can be replaced by a smaller one at the cost of introducing an unknown multiplicative constant.
\begin{proof}[Proof of $ c(\Z^n, k) \le (3n)^{5n} \cdotp k^{\frac{n-1}{n+1}} $ in Theorem~\ref{thmLowerAndUpperBoundIntegers}]
	Since the right-hand side of the asserted inequality is non-decreasing in $k$ and $c(\Z^n,k)$ can be bounded by $g(\Z^n,0),\ldots,g(\Z^n,k)$, it suffices to prove the inequality with $g(\Z^n,k)$ in place of $c(\Z^n,k)$. Let $ \beta(n) := (3n)^{5n} $.
    We proceed by induction on~$n$.
    The inequality is trivial for $ n = 1 $. Let $n \ge 2$ and assume that the inequality has been verified in dimensions $1,\ldots,n-1$.
    Let $ P \subseteq \R^n $ be a polytope with $ \vertices(P) \subseteq
    \Z^n $ and $ \card{\nonverts{P}{\Z^n}} = k $.
    We have to show
    \begin{equation} \label{eqDesiredIneq}
	    \card{\vertices(P)} \le \beta(n) \cdotp k^{\frac{n-1}{n+1}}.
    \end{equation}
    As~$ \beta(n) \cdotp k^{\frac{n-1}{n+1}} $ is non-decreasing in $ n $, in the case $\dim(P) < n$, \eqref{eqDesiredIneq} follows from the induction assumption. Thus, we can assume $\dim(P)=n$.

    It is known and not hard to prove that there exists a point $ x \in \interior(P) $ with $ \ac(P,x) \le n $; see \cite{Gruenbaum63}.
    For example, one can take $x$ to be the center of the maximum volume simplex contained in~$K$.
    We first consider the case that $ P $ contains a ``deep'' integer point, that is, there exists a point $ y \in
    (\tfrac{1}{2}x + \tfrac{1}{2}P) \cap \Z^n $.
    Note that $ y $ is contained in the interior of $ P $ and that it satisfies $ \ac(P,y) \le 2n + 1 $ since
    \begin{align*}
        y + \tfrac{1}{2(n+1)}(P - P) & \subseteq \tfrac{1}{2}x + \tfrac{1}{2}P + \tfrac{1}{2(n+1)}(P - P)
        \\ & = \tfrac{1}{2} P + \tfrac{1}{2} \big(x + \tfrac{1}{n+1}(P - P)\big) \\ & \subseteq
        \tfrac{1}{2} P + \tfrac{1}{2} P = P.
    \end{align*}
    Thus, by Theorem~\ref{thmMahler}, we obtain
    \[
        \vol(P) \le (2n + 2)^n \cdotp \card{\Z^n\cap\interior(P)} \le (2n+2)^n \cdotp k \le (3n)^n \cdotp k,
    \]
    and \eqref{eqDesiredIneq} follows in view of Theorem~\ref{thmAndrews}.

    We are left with the case that $ \frac{1}{2}x + \frac{1}{2}P $ contains no integer point.
    Using Theorem~\ref{thmFlatness} (and its notation), we derive that for some vector $ u \in \Z^n \setminus\{0\}$ one has
    \[
        w(P,u) = 2 \cdot w(\tfrac{1}{2}x + \tfrac{1}{2}P, u) \le 2 \cdot \flt(n).
    \]
    Thus, the set
    \[
	    I:= \setcond{i \in \Z}{ \min_{p \in P} \sprod{p}{u} \le i \le  \max_{p \in P} \sprod{p}{u}}
    \]
    satisfies $|I| \le 2 \cdot \flt(n) + 1$.

    For each $ i \in I $ define $ P_i := \conv( \setcond{p \in P \cap \Z^n}{\sprod{p}{u} = i}) $ and $ k_i :=
    \card{\nonverts{P_i}{\Z^n}} $.
    Note that
    \[
        \nverts{P} \le \sum_{i \in I} \,\nverts{P_i}.
    \]
    Using the induction assumption, we get
    \[
        \nverts{P_i} \le g(n - 1, k_i) \le \beta(n - 1) \cdotp k_i^{\frac{n-2}{n}} \le \beta(n-1) k^{\frac{n-2}{n}}
    \]
    for each $ i \in I $ with $ k_i \ge 1 $, while $ \nverts{P_i} \le 2^{n-1} $ holds for all $ i \in I $ with $ k_i
    = 0 $. This yields $\nverts{P_i} \le (\beta(n-1) + 2^ {n-1} ) k^{\frac{n-2}{n}}$ for every $i \in I$. Consequently,
    \[
	    \nverts{P} \le |I| (\beta(n-1) + 2^ {n-1} ) k^{\frac{n-2}{n}} \le (2\cdot\phi(n) + 1)(\beta(n-1) + 2^ {n-1} ) k^{\frac{n-2}{n}}.
    \]

    \noindent The latter implies \eqref{eqDesiredIneq} since one has $ \frac{n - 2}{n} \le \frac{n - 1}{n + 1} $ and
    \[
        (2 \cdot \flt(n) + 1) \cdotp (\beta(n-1) + 2^{n-1}) \le 3 n^{5/2} \cdotp 2 \beta(n-1)
        \le 6n^3 \cdotp (3n)^{5n - 5} \le (3n)^{5n} = \beta(n).\qedhere
    \]
\end{proof}

\noindent
We now focus on the lower bound in Theorem~\ref{thmLowerAndUpperBoundIntegers}.
%
\begin{proof}[Proof of $ \big\lfloor \left(k/(2n)\right)^{\frac{1}{n+1}} \big\rfloor^{n-1} \le c(\Z^n, k) $ in Theorem~\ref{thmLowerAndUpperBoundIntegers}]
    Consider two parameters $t,s \in \N$, which will be fixed later.
    We compute the number of vertices and the number of integer non-vertices of the integral polytope $P:= \conv(C \cap
    \Z^n)$, where $C$ is a compact convex set given by
    \[
        C := \setcond{(x_1,\dotsc,x_n) \in [1,t]^{n-1} \times \R}{\ell(x_1,\dotsc,x_{n - 1}) \le x_n
            \le u(x_1,\dotsc,x_{n - 1})}
    \]
    with
    \begin{align*}
        \ell(x_1,\dotsc,x_{n-1}) & := \sum_{i=1}^{n-1} (x_i^2 - t^2),
        \\ u(x_1,\dotsc,x_{n-1}) &:= s + \sum_{i=1}^{n-1} (t^2 - x_i^2).
    \end{align*}
    Note that on $ [1,t]^{n-1} $ the function $ \ell $ is strictly smaller than $ u $, the function $ \ell $ is strictly
    convex and the function $ u $ is strictly concave.
    In view of this, a point $(x_1,\dotsc,x_n)$ is a vertex of $ P $ if and only if $ x_1,\dotsc,x_{n - 1} $ all belong
    to $ \{1,\dotsc,t\} $ and $ x_n $ is either equal to $ \ell(x_1,\dotsc,x_{n-1}) $ or equal to $
    u(x_1,\dotsc,x_{n-1}) $.
    Thus, we obtain
    \[
        \nverts{P} = 2 t^{n-1}.
    \]
    Furthermore, a point $ (x_1,\dotsc,x_n)$ belongs to $ P \cap \Z^n $ if and only if $ x_1,\dotsc,x_{n - 1} $ all
    belong to $ \{1,\dotsc,t\} $ and $ x_n $ is an integer value in $ \{\ell(x_1,\dotsc,x_{n - 1}),\dotsc,
    u(x_1,\dotsc,x_{n - 1})\}$.
    Thus, we see that
    \[
        \card{P \cap \Z^n} = \sum_{x_1,\dotsc,x_{n-1} \in [t]} (u(x_1,\dotsc,x_{n-1}) - \ell(x_1,\dotsc,x_{n-1}) + 1),
    \]
    where
    \begin{align*}
        u(x_1,\dotsc,x_{n-1}) - \ell(x_1,\dotsc,x_{n-1}) + 1 = s+ 1 +2 (n-1)  t^2 - 2 \sum_{i=1}^{n-1} x_i^2.
    \end{align*}
    For the determination of $\card{\nonverts{P}{\Z^n}}$ we first note that
    \begin{align*}
        \sum_{x_1,\dotsc,x_{n-1} \in [t]} \sum_{i=1}^{n-1} x_i^2
        & = (n-1) \sum_{x_1,\dotsc,x_{n-1} \in [t]} x_1^2
        \\ & = (n-1) t^{n-2} \sum_{x_1 \in [t]} x_1^2
        \\ & = (n-1) t^{n-2} \cdot \frac{1}{6} t (t+1) (2 t + 1)
        \\ & = \frac{1}{6} (n-1) (t+1) (2 t+1) t^{n-1}.
    \end{align*}
    Hence, we have
    \begin{align*}
        \card{\nonverts{P}{\Z^n}}
        & = \card{P \cap \Z^n} - \nverts{P}
        \\ & = (s + 2 (n-1) t^2 - \frac{1}{3} (n-1) (t+1) (2t+1) - 1) t^{n-1}
        \\ & = \left( s  + \frac{1}{3} (n-1) (4 t^2 - 3 t - 1)  -1 \right) t^{n-1}.
    \end{align*}
    In particular, for $s=1$ one has $\card{\nonverts{P}{\Z^n}} \le 2  n t^{n+1}$.

    Next, we fix $t$ and $s$.
    We choose $t$ to be the largest integer with $2 n  t^{n+1} \le k$, that is,
    \[
        t:=  \floor{\left(\frac{k}{2n}\right)^{\frac{1}{n+1}}}.
    \]
    For the rest of the proof we assume that $k \ge 2n$, since otherwise the asserted inequality is trivial.
    By the choice of $t$, for $s=1$ we have $\card{\nonverts{P}{\Z^n}} \le k$.
    Let $s$ be the largest possible integer, for which the inequality $\card{\nonverts{P}{\Z^n}} \le k$ is
    fulfilled.
    Consider the cardinality
    \[
        k' := \left( s  + \frac{1}{3} (n-1) (4 t^2 - 3 t - 1)  -1 \right) t^{n-1}
    \]
    of $\card{\nonverts{P}{\Z^n}}$.
    Clearly, we have $k' \le k$ by construction.
    Furthermore, $k-k' \le t^{n-1}$.
    In fact, if we had $k-k' > t^{n-1}$, the parameter $s$ could be enlarged by $1$, contradicting the choice of $s$.
    Thus, by construction we have $ k - k' \le t^{n-1}$ and $ g(\Z^n,k') \ge 2 t^{n-1}$, and hence, using
    Theorem~\ref{thmPolytopalDescr}, we get
    \[
        c(\Z^n,k) \ge g(\Z^n,k') + k' - k \ge 2t^{n-1} - t^{n-1} = t^{n-1}.\qedhere
    \]
\end{proof}

%% file: sec-specific-values.tex
\section{Specific values for the integer lattice}
\label{sectSpecificValues}

In this section, we prove Theorem~\ref{thmSpecialValues}. We determine the values $c(\Z^n,k)$ for $k \in \{1,\ldots,4\}$ by computing $g(\Z^n,k)$ for all $k$ in this range. We start by providing lower bounds on $g(\Z^n,k)$: 

\begin{lem} \label{lemGLowerBounds}
	Let $n, k \in \N$. Then $g(\Z^n,k) \ge 2^{n+1} - 2$. Furthermore, if $n \ge 2$, one has $g(\Z^n,4) \ge 2^{n+1}$. 
\end{lem}
\begin{proof}
	It is straightforward to check that the integral polytope 
	\[
		P_n :=\conv \bigl( [-1,0]^n \cup [0,1]^n \bigr)
	\]
	satisfies $\nverts{P_n} = 2^{n+1} -2 $ and $\Z^n \cap P_n \setminus \vertices(P_n) = \{0\}$. This verifies $g(\Z^n,k) \ge 2^{n+1} - 2$ for $k=1$ and every $n \in \N$. Consider $k \ge 2$. If $n=1$, the assertion is trivial. So, let $n \ge 2$. The polytope $P_{n-1}$ generates the prism $P_{n-1} \times [0,1]$ which has exactly two non-vertex integer points, namely $0$ and $u:=(0,\ldots,0,1)$. Adding the integer point $-u$ below $0$ and the integer points $2 \cdot u,\ldots, k \cdot u$ above $u$ we generate the integral polytope 
	\[
		P = \conv \bigl( (P_{n-1} \times [0,1]) \cup (\{0\}^{n-1} \times [-1,k]) \bigr)
	\]
	with $\nverts{P} = 2 (2^n - 2) + 2 =  2^{n+1} - 2$ and $k$ non-vertex integer points $0 \cdot u,\ldots, (k-1) \cdot u$. This concludes the proof of $g(\Z^n,k) \ge 2^{n+1} - 2$ for any $k\in\N$.
	
	For bounding $g(\Z^n,4)$ we can now rely on the existence of a polytope $P' \in \cP(\Z^{n-1})$ with $\nverts{P'} = 2^n - 2$ and with $X' := \Z^{n-1} \cap P'\setminus \vertices(P')$ satisfying $|X'| =2$ (just use the above considerations for $k=2$). Similarly to the previous construction, we build the prism $P'\times [0,1]$ and add the points of $X' \times \{-1\}$ below $X' \times \{0\}$ and the points of $X'\times \{2\}$ above $X'\times \{1\}$. As a result, we obtain the polytope 
	\[
		P = \conv \bigl( (P' \times [0,1]) \cup (X' \times [-1,2]) \bigr) \in \cP(\Z^n)
	\]
	satisfying $\nverts{P} = 2 (2^n -2) + 4 = 2^{n+1}$ and $\card{\Z^n \cap P \setminus \vertices(P)} = 4$. 
\end{proof}

	
	

\begin{proof}[Proof of Theorem~\ref{thmSpecialValues}]
The equality $c(\Z^n,0)=2^n$ is precisely Doignon's theorem~\cite{Doignon1973}.
Since $g(\Z^n,k)$ is non-decreasing for $k \in \{0,\ldots,4\}$, we obtain, in view of \eqref{boundsCthrougG}, that $c(\Z^n,k) = g(\Z^n,k)$ holds for every $k \in \{1,\ldots,4\}$. 
Thus it suffices to determine $g(\Z^n,k)$ for $k \in \{1,\ldots,4\}$.

The desired lower bounds on $g(\Z^n,k)$ are provided by Lemma~\ref{lemGLowerBounds}. We need to verify the matching upper bounds. 
 To this end, consider an arbitrary polytope $P \in \cP(\Z^n)$ such that the set $X:=\Z^n \cap P \setminus \vertices(P)$ satisfies $\card{X} = k$. We introduce the polytope $Q:= \conv (X)$.

Let $t$ be the number of vertices of $Q$. Since $1 \le k \le 4$, one has $1 \le t \le 4$. If $t \le 3$, then $\dim(Q) \le 2$ and Lemma~\ref{lemUpperBoundGeneralDataQ} yields $\nverts{P} \le 2^{n+1} - 2$. For $k \le 3$ one has $t \le 3$, so that the latter considerations already imply $g(\Z^n,k) \le 2^{n+1} - 2$.

To get the desired upper bound for $k=4$, the further possible value $t=4$ needs to be addressed. For $t=4$, the polytope $Q$ is either a quadrilateral or a tetrahedron. If $Q$ is a quadrilateral, Lemma~\ref{lemUpperBoundGeneralDataQ} yields $\nverts{P} \le 2^{n+1}$. 
 
 For the rest of the proof, let $Q$ be a tetrahedron. Observe that $\Z^n \cap Q = X$. We use partitioning of $\Z^n$ into $2^n$ residue classes modulo $2$.  Two points $x \in \Z^n$ and $y \in\Z^n$ are said to be in the same residue class modulo $2$ if $x-y\in 2\Z^n$. Indexing the residue classes by $i \in [2^n]$, we denote by $V_i$ the set of all vertices of $P$ that fall into the $i$-th class. 
 
 With each $V_i$ we associate the set $M_i := \setcond{(v+w)/2}{v, w \in V_i, \ v \ne w}$ consisting of all midpoints between pairs of distinct points in $V_i$. By the choice of $V_i$, one has $M_i \subseteq X$.  In what follows, we bound the cardinalities of the sets $V_i$ to get the desired bound on $\nverts{P}$. 
 
First of all, every $V_i$ contains at most three points, for otherwise $M_i$ would have cardinality at least five, which is a contradiction to $|X| = k = 4$. 
 
 If the elements of $V_i$ are congruent modulo $2$ to some vertex of $Q$, we even get $|V_i| \le 1$. In fact, assume that $V_i$  contains two distinct points $u,v$ and let $w$ be the vertex of $Q$ belonging to the same residue class as $u$ and $v$. Then, the convex hull of $u, v$ and $w$ is either a line segment containing at least three integer points in its relative interior, or it is a triangle whose vertex~$u$ and the three midpoints of the edges are integer points. Thus, we find either three non-vertex integer points on a line or four non-vertex integer points in a two-dimensional affine space. Both cases contradict the properties of~$Q$.
 
 Finally, observe that there are at most four sets $V_i$ with exactly three points. In fact, if~$V_i$ consists of exactly three points, then $\conv(V_i)$ is a triangle with integer vertices. This implies that $\conv(M_i)$ is a triangle with integer vertices, too. Since $M_i \subseteq X$, we conclude that $\conv(M_i)$ is a facet of $Q$. Taking into account that a triangle is uniquely determined by its edge midpoints and that $Q$ has four facets, we get that there are at most four sets $V_i$ with $|V_i| = 3$. 
 
Summarizing, we obtain that four of the sets $V_{i}$ have cardinality at most $1$, at most four of the sets~$V_i$ have cardinality $3$ and all remaining sets have cardinality $2$. This yields
 \[
	 \nverts{P} = \sum_{i=1}^{2^n} \card{V_i} \le 4 \cdot 1 + 4 \cdot 3 + (2^n - 8) \cdot 2 = 2^{n+1}
 \]
 and concludes the proof. 
\end{proof}

%% file: sec-maximal-sets.tex
\section{Maximal convex sets with \texorpdfstring{$k$}{k} points of \texorpdfstring{$S$}{S} in the interior}
\label{sectMaximalFree}

\newcommand{\cM}{\mathcal{M}}

This section presents another interpretation of $c(S,k)$ in the case that $S \subseteq \R^n$ is discrete and $k \in \N_0$ is arbitrary.  We introduce the family $\cM(S,k)$ of all inclusion-maximal $n$-dimensional convex sets with precisely $k$ points of $S$ in the interior. More formally, $M \in \cM(S,k)$ if and only if $M$ is an $n$-dimensional convex set with $\card{\interior(M) \cap S}=k$ such that for every convex set $M'$ satisfying $M \subseteq M'$ and $\interior(M') \cap S = \interior(M) \cap S$ one necessarily has $M = M'$.

In optimization, sets $M \in \cM(S,0)$ are called \emph{maximal $S$-free}; see \cite{Averkov13}, \cite[Sect.~2]{BasuConfortiCornuejols2010} and~\cite{MoranDey2011}.
Various families of maximal $\Z^n$-free sets have been extensively used for the generation of so-called \emph{cutting planes}; see the survey \cite{ConfortiCornuejolsZambelli2011}.
Cutting planes are employed for gradually approximating a given mixed-integer problem with linear optimization problems; they belong to the standard tools for solving general mixed-integer problems. Note that, apart from $S=\Z^n$, also other choices of $S$ such as $S= \N_0^n$ are of interest for mixed-integer optimization; see  \cite{FukasawaGuenluek2011}. The possibility of using more general sets $M \in \cM(S,k)$, where $k \in \N_0$ is arbitrary, similarly to maximal $S$-free sets is supported by the following simple observation.
If $S = \Z^n$ and it is known that a particular point $z \in \Z^n$ does not correspond to a solution of the underlying mixed-integer problem, one can use sets $M \in \cM(S,k)$ with $k=1$  
and $\interior(M) \cap \Z^n = \{z\}$ for the generation of cutting planes. Sets $M \in \cM(S,k)$ for larger parameters $k$ can be used analogously.  

From the computational point of view, the complexity of generating cutting planes from $M \in \cM(S,k)$ depends on the number $\card{\facets(M)}$ of facets of $M$ (where we interpret $\card{\facets(M)}$ as $\infty$ if $M$ is not a polyhedron). Thus, describing the maximum of $\card{\facets(M)}$ for $M \in \cM(S,k)$ is of interest. Theorem~\ref{thmMaxSets} below shows that this maximum is exactly $c(S,k)$. In the proof of this result we use the following proposition:

\begin{prop}
	\label{propEnlarging}
	Let $S \subseteq \R^n$ be discrete, $k \in \N_0$ and let $C$ be an $n$-dimensional convex set with $\card{\interior(C) \cap S}=k$. Then there exists an $M \in \cM(S,k)$ with $C \subseteq M$. 
\end{prop}
\begin{proof}
	The assertion can be derived by a direct application of Zorn's lemma. For $k=0$ this was observed in \cite[Sect.~2]{BasuConfortiCornuejols2010} and \cite[Sect.~3]{MoranDey2011}. 
\end{proof}

The following is a somewhat stronger version of Theorem~\ref{thmMaximalSets} from the introduction.

\begin{thm} \label{thmMaxSets}
	Let $S \subseteq \R^n$ be discrete and let $k \in \N_0$. Then
	\[
		c(S,k) = \max \setcond{\card{\facets(M)}}{M \in \cM(S,k)},
	\]
	where we interpret $\card{\facets(M)}$ as $\infty$ if $M$ is not a polyhedron. 
\end{thm}
\begin{proof}
	Let $m:= \max \setcond{\card{\facets(M)}}{M \in \cM(S,k)}$. We need to verify $c(S,k) = m$. If $c(S,k) = - \infty$, there are no convex sets that contain exactly $k$ points of $S$. This implies $\cM(S,k) = \emptyset$ so that both $c(S,k)$ and $m$ are equal to $-\infty$. Let now $c(S,k) > -\infty$.
	
	We verify $c(S,k) \ge m$. This inequality is trivial if $c(S,k) = \infty$ or $m = -\infty$. So, assume that $c(S,k)$ is finite and $m > - \infty$. Consider an arbitrary $M \in \cM(S,k)$. We show $\card{\facets(M)} \le c(S,k)$ using a  straightforward adaption of the argument in \cite[Lem.~4.1]{Averkov13} from the case $k=0$ to the case of an arbitrary $k\in\N_0$. Let $X:=\interior(M) \cap S$. There exist  $n$-dimensional polytopes $P_t$ with $t \in \N$ such that $P_t \subseteq P_{t+1}$ and $\interior(P_t) \cap S = X$ for every $t \in \N$ and $\bigcup_{t=1}^{\infty} P_t = \interior(M)$.
	Since $P_t$ is the intersection of finitely many closed halfspaces, by Lemma~\ref{lemRestrictionToClosedHalfSpaces} there exists a polyhedron $Q_t$ with at most $c(S,k)$ facets satisfying $P_t \subseteq Q_t$ and $\interior(Q_t) \cap S = X$. For $t \rightarrow \infty$, the compactness argument from \cite[Proof of Lem.~4.1]{Averkov13} applies 
without any changes: The polyhedra $Q_t$, all having at 
most $c(S,k)$ facets, generate a polyhedron $Q$ with at most $c(S,k)$  facets that satisfies $\interior(Q) \cap S = X$ and $M \subseteq Q$. Since~$M$ is maximal, one must have $M = Q$. That is, $M$ is a polyhedron. 
	
	It remains to verify $c(S,k) \le m$. 
	This is trivial if $m=\infty$.
	So, we assume that $m<\infty$ which means in particular that every set in $\cM(S,k)$ is a polyhedron.
	In view of \eqref{eqPolytInterp}, it suffices to show that $t:=\card{S \cap P} -k \le m$ for every $P$ as in \eqref{eqPolytInterp}. Clearly, $0 \le t \le \nverts{P}$. We fix arbitrary vertices $v_1,\ldots,v_t \in S$ of $P$. Since $S$ is discrete, we can enclose $P$ into an $n$-dimensional polytope $Q$ such that $P \cap S = Q \cap S$, $\bd(Q) \cap S = \{v_1,\ldots,v_t\}$ and $v_1,\ldots,v_t$ lie in the relative interior of pairwise distinct facets $F_1,\ldots,F_t$ of $Q$. By construction, $\interior(Q) \cap S = \card{S \cap P} - t = k$. Proposition~\ref{propEnlarging} yields the existence of $M \in \cM(S,k)$ with $Q \subseteq M$. Since $v_1,\ldots,v_t \in S$ lie in the relative interior of pairwise distinct facets $F_1,\ldots,F_t$ of $Q$, these facets are subsets of pairwise distinct facets of $M$. This shows that~$M$ has at least $t$ facets and concludes the proof of $c(S,k) \le m$. 
\end{proof}

Since the choice $S = \Z^n$ is of particular interest, we conclude the section by describing the geometry of the sets $M \in \cM(\Z^n,k)$ more precisely. In view of Theorem~\ref{thmMaxSets} and the fact that $c(\Z^n,k) < \infty$, each $M \in \cM(\Z^n,k)$ is a polyhedron. 

A geometric characterization of the polyhedra in $\cM(\Z^n,0)$ was presented by Lov\'asz in~\cite{Lovasz1989}; see also \cite{AverkovAProof2013,BasuConfortiCornuejols2010} for proofs of this characterization. In particular, it is known that the unbounded polyhedra in $\cM(\Z^n,0)$ can be described through the bounded polyhedra in $\cM(\Z^i,0)$ with $1 \le i < n$. Furthermore, the bounded polyhedra in $\cM(\Z^n,0)$ are precisely those $n$-dimensional polytopes that have no interior integer point and at least one integer point in the relative interior of each facet.  It turns out that for $k \ge 1$, all elements of $\cM(\Z^n,k)$ are polytopes. We show the above description of polytopes in $\cM(\Z^n,0)$ can be carried over to $\cM(\Z^n,k)$ without any changes.

\begin{thm}
	Let $k \in \N$. Then $\cM(\Z^n,k)$ is the set of all $n$-dimensional polytopes $P$ with precisely $k$ interior integer points such that the relative interior of every facet of $P$ contains a point of $\Z^n$. 
\end{thm}
\begin{proof}
	In the proof we adapt the arguments from \cite{AverkovAProof2013}. Let $M \in \cM(\Z^n,k)$. Without loss of generality assume that $0 \in \interior(M)$. 
	
	We first verify the boundedness of $M$ by adapting an argument from \cite[Lem.~4]{AverkovAProof2013}. Assume that $M$ is not bounded. Then there exists a nonzero vector $u$ such that $ \alpha u \in \interior(M)$ for every $\alpha \ge 0$.
	We consider a ball $B$ with center at the origin and such that $B \subseteq \interior(M)$. If $\lambda \ge 0$ is large enough, the $0$-symmetric convex body $K:=\frac{1}{k} (\lambda [-u,u] + B)$ has volume larger than $2^n$. It follows, by Minkowski's first theorem (cf.~\cite[Sect.~22]{Gruber2007}), that $\interior(K)$ contains a nonzero integer vector $z$. Hence $k z \in \interior(\lambda [-u,u] + B)$. After possibly replacing $z$ by $-z$, we can assume $kz \in \interior(\lambda [0,u] + B) \subseteq \interior(M)$. Thus, $0 z,\ldots, k z$ are $k+1$ integer points in $\interior(M)$, which is a contradiction to the choice of $M$. 
	
	Once the boundedness of $M$ is established, the rest of the assertion follows by generalizing the argument from \cite[Proof of Thm.~1]{AverkovAProof2013} from the case $k=0$ to the case of an arbitrary $k$ in a straightforward manner. 
\end{proof}

%% file: sec-appendix.tex
\section{An explicit constant for Andrews' Theorem}
 \label{secAppendixA}
\noindent
Let $ n \in \N $ with $ n \ge 2 $.
Following \citeauthor{Andrews63}' proof and notation, it is shown in~\cite{Andrews63} that for every $ n $-dimensional
polytope $ P \in \cP(\Z^n) $ one has
\[
    \vol(P) \ge \kappa(n) \nverts{P}^{\frac{n+1}{n-1}},
\]
where
\begin{align*}
    \kappa(n) & = \frac{1}{2} \cdotp 3^{-n} \gamma(n) \left( \kappa'(n) \right)^{\frac{n}{n-1}}, & & (\text{p.}~278\text{ in Thm.})\\
    \gamma(n) & = n^{-n} \cdotp c_1(n), & & (\text{p.}~275\text{ in Lem.~7})\\
    \kappa'(n) & = \left( n! \right)^{-\frac{n-1}{n}} \cdotp \left( \xi(n) \right)^{-\frac{n-1}{n}} \cdotp
        \left( \frac{2 \pi^{n/2}}{\Gamma(n/2)} \right)^{-\frac{1}{n}}, & & (\text{p.}~278\text{ in Thm.})\\
    c_1(n) & = \frac{1}{n!} \sqrt{n + 1} \left( \frac{(n-2)!}{\sqrt{n}} \right)^{\frac{n}{n-1}}, & & (\text{p.}~275\text{ in Lem.~7})\\
    \xi(n) & = (n-1)^{2n} \left( \frac{n!}{n^{2n}} \right)^{\frac{1}{n-1}} \left( (n-1)! \right)^{\frac{1}{n-1}}, & & (\text{p.}~274\text{ in Lem.~6})
\end{align*}
and $ \Gamma(\cdot) $ is the Gamma-function; the information in the parentheses gives the exact place where the reader can find the respective constant in~\cite{Andrews63}.
Thus, the constant $ \alpha(n) $ in Theorem~\ref{thmAndrews} can be chosen as any number at least
\begin{equation}
    \label{eqUpperBoundOnAlpha}
    \left( \kappa(n) \right)^{-\frac{n-1}{n+1}}
    = \frac{ 2^{\frac{n-1}{n+1}} (3n)^{\frac{n(n-1)}{n+1}} }{ \left( c_1(n) \right)^{\frac{n-1}{n+1}} \left( \kappa'(n)
    \right)^{\frac{n}{n+1}} }.
\end{equation}
We give an upper bound on this value by using the following simple estimations.
Since
\[
    \xi(n) \le n^{2n} \cdotp \left( \frac{n^n}{n^{2n}} \right)^{\frac{1}{n}} \cdotp \left( (n-1)^{n-1}
    \right)^{\frac{1}{n-1}} =n^{2n} \cdotp \frac{1}{n} \cdotp (n-1) \le n^{2n},
\]
we obtain
\begin{equation}
    \label{eqUpperBoundOnKappaPrime}
    \kappa'(n) = \frac{1}{\left( n! \cdotp \xi(n) \right)^{\frac{n-1}{n}}} \cdotp
    \frac{(\Gamma(n/2))^{\frac{1}{n}}}{\sqrt[n]{2} \sqrt{\pi}}
    \ge \frac{\frac{1}{2}}{n! \cdotp \xi(n) \cdotp 4}
    \ge \frac{1}{8n^{3n}}.
\end{equation}
Furthermore, we clearly have
\begin{equation}
    \label{eqUpperBoundOnCOne}
    c_1(n) \ge \frac{1}{n!} \sqrt{n + 1} \frac{(n-2)!}{\sqrt{n}} \ge \frac{1}{n^2}.
\end{equation}
Thus, plugging~\eqref{eqUpperBoundOnKappaPrime} and~\eqref{eqUpperBoundOnCOne} into~\eqref{eqUpperBoundOnAlpha}, we
established
\begin{align*}
    \left( \kappa(n) \right)^{-\frac{n-1}{n+1}} \le \frac{2^{\frac{n-1}{n+1}}
    (3n)^{\frac{n(n-1)}{n+1}}}{\left(\frac{1}{n^2}\right)^{\frac{n-1}{n+1}}
    \cdotp \left( \frac{1}{8n^{3n}} \right)^{\frac{n}{n+1}}}
    \le 2 \cdotp (3n)^n \cdotp n^2 \cdotp 8n^{3n} \le (3n)^{4n}.
\end{align*}